\documentclass[reqno,10pt]{amsart}
\usepackage{amsmath,amssymb,latexsym,soul,cite,mathrsfs}
\usepackage{color,enumitem,graphicx}
\usepackage[colorlinks=true,urlcolor=blue,
citecolor=red,linkcolor=blue,linktocpage,pdfpagelabels,
bookmarksnumbered,bookmarksopen]{hyperref}
\usepackage[english]{babel}
\usepackage[left=2.7cm,right=2.7cm,top=3.2cm,bottom=3.2cm]{geometry}

\numberwithin{equation}{section}

\newtheorem{theorem}{Theorem}[section]
\newtheorem{lemma}{Lemma}[section]
\newtheorem{corollary}{Corollary}[section]
\newtheorem{proposition}{Proposition}[section]

\newtheorem{remark}{Remark}[section]
\newtheorem{definition}{Definition}[section]

\title[On Lane-Emden systems with singular nonlinearities]
 {On Lane-Emden systems with singular nonlinearities and applications to MEMS}

\author[J.M.\ do \'O]{Jo\~ao Marcos do \'O}
\author[R.G.\ Clemente]{Rodrigo G.\ Clemente}

\address[J.M. do \'O]{Department of Mathematics,
Federal University of Para\'{\i}ba
\newline\indent
58051-900, Jo\~ao Pessoa-PB, Brazil}
\email{\href{mailto:jmbo@pq.cnpq.br}{jmbo@pq.cnpq.br}}

\address[R.\ Clemente]{Department of Mathematics, 
 Rural Federal University of Pernambuco
\newline\indent 
52171-900, Recife, Pernambuco, Brazil}
\email{\href{mailto:rodrigo.clemente@ufrpe.br}{rodrigo.clemente@ufrpe.br}}

\thanks{Research partially supported by the National Institute of Science and Technology of Mathematics INCT-Mat, CAPES and CNPq.}
 
\subjclass[2000]{35J47, 35J75, 35B65}
\keywords{Nonlinear PDE of elliptic type, singular nonlinearity, elliptic systems, semi-stable solution, extremal solution, regularity of extremal solutions.}

\begin{document}

\begin{abstract}
In this paper we analyse the Lane-Emden system
\begin{equation}
\left\{
\begin{alignedat}{3}
-\Delta u = & \, \frac{\lambda f(x)}{(1-v)^2} & \quad \text{in} & \quad\Omega\\
-\Delta v = & \, \frac{\mu g(x)}{(1-u)^2} & \quad \text{in} & \quad\Omega\\
0\leq u &, v < 1 & \quad \text{in} & \quad \Omega\\
u = v & =  \, 0  &  \text{on} & \quad \partial\Omega\\
\end{alignedat}
\right.\tag{$S_{\lambda, \mu}$}
\end{equation}
where $\lambda$ and $\mu$ are positive parameters and $\Omega$ is a smooth bounded domain of $\mathbb{R}^N$ $( N \geq 1)$. Here we prove the existence of a critical curve $\Gamma$ which splits the positive quadrant of the $(\lambda,\mu)\text{-plane}$ into two disjoint sets $\mathcal{O}_1$ and $\mathcal{O}_2$ such that the problem \eqref{MS} has a smooth minimal stable solution $(u_\lambda,v_\mu)$ in $\mathcal{O}_1$, while for $(\lambda,\mu)\in\mathcal{O}_2$ there are no solutions of any kind. We also establish upper and lower estimates for the critical curve $\Gamma$ and regularity results on this curve if $N\leq 7$. Our proof is based on a delicate combination involving maximum principle and $L^p$ estimates for semi-stable solutions of \eqref{MS}.
\end{abstract}

\maketitle

\section{Introduction}
\hspace{0.6cm}In this paper we deal with Hamiltonian systems of coupled singular elliptic equations of second-order of the form 
\begin{equation}\label{MS}
\left\{
\begin{alignedat}{3}
-\Delta u = & \, \frac{\lambda f(x)}{(1-v)^2} & \quad \text{in} & \quad\Omega,\\
-\Delta v = & \, \frac{\mu g(x)}{(1-u)^2} & \quad \text{in} & \quad\Omega,\\
0\leq u &, v < 1 & \quad \text{in} & \quad \Omega,\\
u = v & =  \, 0  &  \text{on} & \quad \partial\Omega,\\
\end{alignedat}
\right.\tag{$S_{\lambda, \mu}$}
\end{equation}
where $\lambda$ and $\mu$ are positive parameters, $\Omega$ is a smooth bounded domain of $\mathbb{R}^N  \;(N \geq 2)$ and $f$ and $g$ satisfy the following conditions:
\begin{equation}\label{05}\tag{$H$}
\begin{alignedat}{1}
f,g\in C^\alpha(\overline{\Omega})\text{ for some }\alpha\in (0,1]\text{, }0\leq f,g\leq 1 \text{ and }\\
f,g>0 \text{ on a subset of }\Omega\text{ of positive measure.}
\end{alignedat}
\end{equation}

\subsection{Motivation and related results}
\hspace{0.6cm}System \eqref{MS} can be seen as a Lane-Emden type system with nonlinearities with negative exponents \cite{E,L,SZ2,S}. A lot of work has been devoted to existence and nonexistence of solutions to elliptic systems with continuous power like nonlinearities, among which we recall \cite{DOR,DOR2,DFF,DFM,DFS,HV,MP} and the survey \cite{DF}, just recently Lane-Emden type singular nonlinearities have been considered in \cite{G}. Here we address the problem of studying existence, non-existence and regularity results by means of the nonlinear eigenvalue problem \eqref{MS}, in which for the sake of clarity we consider a Coulomb nonlinear source though most results extend to more general situations. Related results for systems with continuous nonlinearities have been obtained in \cite{H,MM}.

Another important motivation to consider \eqref{MS} comes from recent works on the study of the equations that models MEMS
\begin{equation}\label{E1}\tag{$P_\lambda$}
\begin{cases}
-\Delta v  =  \lambda \, \frac{g(x)}{(1-v)^2} \quad & \text{ in }  \Omega, \\
0\leq v<1    & \text{ in }  \Omega,  \\
v = 0 & \text{ on }\partial \Omega.
\end{cases}
\end{equation}
Micro-electromechanical systems (MEMS) are often used to combine electronics with micro size mechanical devices in the design of various types of microscopic machinery. MEMS devices have therefore become key components of many commercial systems, including accelerometers for airbag deployment in vehicles, ink jet printer heads, optical switches and chemical sensors.

Nonlinear interaction described in terms of coupling of semilinear elliptic equations has revealed through the last decades a fundamental tool in studying nonlinear phenomena, see e.g. \cite{CDFM,DP,DFF,DFM,DH} and references therein. In all the above contexts the nonlinearity is fairly represented by a continuous function. More recently, a rigorous mathematical approach in modeling and designing Micro Electro Mechanical Systems has demanded the need to consider also nonlinearities which develop singularities. In a nutshell, one may think of MEMS' actuation as governed by the dynamic of a micro plate which deflects towards a fixed plate, under the effect of a Coulomb force, once that a drop voltage is applied.

In the stationary case, the naive model which describes this device cast into the second order elliptic PDE \eqref{E1}, where $\Omega$ is a bounded smooth domain in $\mathbb{R}^{N}$ and the positive function $g$ is bounded and related to dielectric properties of the material, see the survey \cite{EGG} and also \cite{GHOGUO,K,PB} for more technical aspects. The key feature of the equation in \eqref{E1} is retained by the discontinuity of the nonlinearity which blows up as $v\to 1^-$ and this corresponds in applications to a snap through of the device.

The general goal on the study of \eqref{E1} is to analyse the structure of the branch of solutions as well as their qualitative properties. The role of the positive parameter $\lambda$ is that of tuning the drop voltage, whence from the PDE point of view, yields the threshold between existence and non-existence of solutions which exist up to a maximal value $\lambda^*$. This is referred in literature as the regularity issue for extremal solutions, see for instance \cite{CRT,EGG,GW,PS2}.

Here we mention some recent papers on semilinear elliptic system of cooperative type which are closely related with our work. M.~Montenegro in \cite{MM} studied elliptic systems of the form $\Delta u=\lambda f(x,u,v)$, $\Delta v=\mu g(x,u,v)$ defined in $\Omega$ a smooth bounded domain under homogeneous Dirichlet boundary conditions. Under some suitable assumptions, which include in particular that the systems are cooperative, it was proved that there exists a monotone continuous curve $\Upsilon$ in the positive quadrant $\mathcal{Q}$ separating this set into two connected components: $U$ ``below'' $\Upsilon$, where there are $C^1(\overline{\Omega})$ minimal positive solutions, and $V$ ``above'' $\Upsilon$, where there is no such solution. For points on $\Upsilon$ there are weak solutions (in the sense of the weighted Lebesgue space $L^1_d(\Omega )$ , where $d(x)$ is the distance to the boundary $\partial \Omega$. Linearized stability of solutions in $U$ is also proved. The existence proof uses sub- and supersolutions, and the existence of weak solutions is shown by a limiting argument involving a priori estimates in $L^1_d(\Omega)$ for classical solutions.

A question that attracted a lot of attention is the regularity of the extremal solution. For the scalar case \eqref{E1}, F.~Mignot and J-P.~Puel \cite{MIGPUE1980} studied regularity results to certain nonlinearities, namely,  $g(u)=e^u$, $g(u)=u^m$ with $m>1$, $g(u)=1/(1-u)^k$ with $k>0$. Very recently, this analysis was complemented by N. Ghoussoub and Y. Guo \cite{GHOGUO} for the MEMS case in a bounded domain $\Omega$ under zero Dirichlet boundary condition, among other refined properties for stable steady states they proved that extremal solutions are smooth if $1\leq N\leq 7$ and $N=8$ is the critical dimension for this class of problems.

For elliptic systems, the stability inequality was first established in the study of Liouville theorems and De Giorgi's conjecture for systems, see \cite{FAZGHO}. There is a correspondence between regularity of extremal solutions and Liouville theorems up to blow up analysis and scaling. This inequality was used to establish regularity results in \cite{COWFAZ2014} for systems and \cite{COGH2014} for the fourth order case. C.~Cowan \cite{COW2011} considered the particular case of nonlinearities of Gelfand type, that is, when $f(x,u,v)=e^v$ and $g(x,u,v)=e^u$. He studied the regularity of the extremal solutions on the critical curve, precisely, he proved that if $3\leq N\leq 9$ and $(N-2)/8<\mu^*/\lambda^*<8/(N-2)$ then the associated extremal solutions are smooth. This implies that $N=10$ is the critical dimension for the Gelfand systems, because the scalar equation related with this class of problems may be singular if $N\geq 10$. Later, C.~Cowan and M.~Fazly in \cite{COWFAZ2014} examined the elliptic systems given by 
\begin{equation}\label{11}
-\Delta u=\lambda f^\prime(u)g(v) ,\quad -\Delta v=\mu f(u)g^\prime(v) \text{ in }\Omega,
\end{equation} 
and
\begin{equation}\label{12}
-\Delta u=\lambda f(u)g^\prime(v),\quad -\Delta v=\mu f^\prime(u)g(v)\text{ in }\Omega
\end{equation}
with zero Dirichlet boundary condition in a bounded convex domain $\Omega$. They proved that for a general nonlinearities $f$ and $g$ the extremal solution associated with \eqref{11} are bounded when $N\leq 3$. For a radial domain, they proved the extremal solution are bounded provided that $N<10$. The extremal solution associated with \eqref{12} are bounded in the case where $f$ is a general nonlinearity and $g(v)=(1+v)^q$ for $1<q<+\infty$ and $N\leq 3$. For the explicit nonlinearities of the form $f(u)=(1+u)^p$ and $g(v)=(1+v)^q$ certain regularity results are also obtained in higher dimensions for \eqref{11} and \eqref{12}.

In the recent years, this class of problems has two natural fourth order generalizations and extensions. D.~Cassani, J.~M.~do \'{O} and N.~Ghoussoub in \cite{COG} considered the problem
\begin{equation}\label{13}
\left\{
\begin{alignedat}{3}
\Delta^2 u = & \, \frac{\lambda f(x)}{(1-v)^2} & \quad \text{in} & \quad\Omega,\\
0\leq u <& 1 & \quad \text{in} & \quad \Omega,\\
u = \frac{\partial u}{\partial \eta} & =  \, 0  &  \text{on} & \quad \partial\Omega
\end{alignedat}
\right.
\end{equation}
with the biharmonic operator $\Delta^2$ and subject to Dirichlet conditions where $\eta$ denotes the outward pointing unit normal to $\partial\Omega$. In the physical model, they consider the plate situation in which flexural rigidity is now allowed whose effects however dominates over the stretching tension, neglecting non-local contributions. Since there is no maximum principle for $\Delta^2$ with Dirichlet boundary conditions for general domains, the authors exploit the positivity of the Green function due to T.~Boggio \cite{BOG} and consider problem \eqref{13} restrict to the ball. After that, C.~Cowan and N.~Ghoussoub \cite{COGH2014} studied the fourth order problem
\begin{equation}\label{14}
\left\{
\begin{alignedat}{3}
\Delta^2 u = & \, \lambda f(u) & \quad \text{in} & \quad\Omega,\\
0\leq u <& 1 & \quad \text{in} & \quad \Omega,\\
u=\Delta u & =  \, 0  &  \text{on} & \quad \partial\Omega
\end{alignedat}
\right.
\end{equation}
with Navier boundary conditions where $f$ is one of following nonlinearities: $f(u)=e^u$, $f(u)=(1+u)^p$ or $f(u)=(1-u)^{-p}$. Note that one can view the fourth order equation \eqref{14} as a system of the following type
\begin{equation}\label{15}
\left\{
\begin{alignedat}{3}
-\Delta v = & \, \lambda f(u) & \quad \text{in} & \quad\Omega,\\
-\Delta u = & \, v & \quad \text{on} & \quad\partial\Omega,\\
u= v & =  \, 0  &  \text{on} & \quad \partial\Omega.
\end{alignedat}
\right.
\end{equation}
Using this approach, they proved regularity results for semi-stable solutions and hence for the extremal solutions using a stability inequality obtained for the elliptic system \eqref{15} associated with the problem \eqref{14}. 

\subsection{Statement of main results} 
\hspace{0.6cm}The main goal of this article is to provide a supplement for the ongoing studies of nonlinear eigenvalue problems of MEMS type, as this is the case for references \cite{COW2011,COWFAZ2014,MM}. Our first result deals with the existence of a curve that split the positive quadrant into two connected components.

\begin{theorem}\label{04}
	Suppose that condition \eqref{05} holds. Then, there exists a curve $\Gamma$ that separates the positive quadrant $\mathcal{Q}$ of the $(\lambda,\mu)$-plane into two connected components $\mathcal{O}_1$ and $\mathcal{O}_2$. For $(\lambda,\mu)\in \mathcal{O}_1$, problem \eqref{MS} has a positive classical minimal solution $(u_\lambda,v_\lambda)$. Otherwise, if $(\lambda,\mu)\in \mathcal{O}_2$, there are no solutions.
\end{theorem}

\hspace{0.6cm}The Theorem \ref{06} and Theorem \ref{07} contain upper and lower estimates for the critical curve. These estimates depend only on $f,g, |\Omega|$ and the dimension $N$, namely,

\begin{theorem}\label{06}
	Suppose f,\,g satisfy \eqref{05}. Then the region $\mathcal{O}_1$ is nonempty, more precisely, there exist a positive constant $C_N$ which depends only of the dimension $N$ such that
	\[
	(0,a_{(f,|\Omega|,N)}]\times(0,a_{(g,|\Omega|,N)}] \subset \mathcal{O}_1,
	\]
	where
	\[
	a_{(f,|\Omega|,N)}:=C_{N}\frac{1}{\sup_\Omega f(x)}\left(\frac{\omega_N}{|\Omega|}\right)^{2/N},\quad a_{(g,R,N)}:=C_{N}\frac{1}{\sup_\Omega
		g(x)}\left(\frac{\omega_N}{|\Omega|}\right)^{2/N}
	\]
	and
	\[
	C_{N}=\max\left\{\frac{8N}{27},\frac{6N-8}{9} \right\}.
	\]
\end{theorem}

\begin{theorem}\label{07}
	Suppose f,\,g satisfy \eqref{05}. Assume that $\inf_\Omega f(x)>0$ (respectively $\inf_\Omega g(x)>0$), then
	\[
	\lambda^*\leq \frac{4\mu_1}{27}\frac{1}{\inf_\Omega f(x)} \; \left(\mbox{
		respectively }\;\; \mu^* \leq \frac{4\mu_1}{27}\frac{1}{\inf_\Omega
		g(x)}\right),
	\]
	where $\mu_1$ is the first eigenvalue of $(-\Delta,\, H_{0}^{1}(\Omega))$. Therefore, if $\inf_\Omega f(x)>0$ and $\inf_\Omega g(x)>0$ the region $\mathcal{O}_1$ is bounded, precisely,
	\[
	\mathcal{O}_1 \subset \left(0,\frac{4\mu_1}{27}\frac{1}{\inf_\Omega
		f(x)}\right)\times \left(0,\frac{4\mu_1}{27}\frac{1}{\inf_\Omega
		g(x)}\right).
	\]
\end{theorem}

\hspace{0.6cm}In the next two theorems we discuss the monotonicity properties of the critical curve for system \eqref{MS}. We mention that similar results have been proved for the scalar case \eqref{E1} in \cite{EGG2007, GHOGUO}. In \cite{EGG2007}, it was shown that the permittivity profile $g$ can be change the bifurcation diagram and alter the critical dimension for compactness for the equation \eqref{E1}.

\begin{theorem}\label{pitomba}
	Suppose that condition \eqref{05} holds. If (\ref{MS}) has a solution in $\Omega$, then it also has a solution for any subdomain $\Omega'\subset \Omega$ for which the Green's function exists. Furthermore, $\lambda^*(\Omega')\geq\lambda^*(\Omega)$ and for the corresponding minimal solutions, we have $u_{\Omega'}(x)\leq u_{\Omega}(x)$ and $v_{\Omega'}(x)) \leq v_{\Omega}(x)$ for all $x\in\Omega$.
\end{theorem}

\begin{theorem}\label{monoto1}
	Let $f,g$ satisfying \eqref{05} and $f^\sharp,g^\sharp$ the Schwarz symmetrization of $f$ and $g$ respectively. Then $\lambda^*(\Omega,f,g)\geq \lambda^*(B_R,f^\sharp,g^\sharp)$ and for each $\lambda \in (0,\lambda^\ast(B_R,f,g))$ we have $\Gamma_{(\Omega,f,g)}(\lambda)\geq\Gamma_{(B_R,f^\sharp,g^\sharp)}(\lambda)$.
\end{theorem}

\hspace{0.6cm}Analogously to scalar case (see \cite{GHOGUO}), we can define the notion of extremal solution of \eqref{MS} for points on the critical curve. Precisely, let us consider a sequence $(\lambda_n,\mu_n)$ below the critical curve converging to a point $(\lambda^*,\mu^*)$ on the critical curve. In view of Theorem \ref{04}, we can consider the minimal solution $(u_{\lambda_n},v_{\mu_n})$ of System $(S_{\lambda_n,\mu_n})$. Now, we can define the extremal solution $(u^*,v^*)$ at $(\lambda^*,\mu^*)$ by passing to the limit when $n\rightarrow +\infty$, namely,
\[
(u^*,v^*)=\lim_{n\rightarrow +\infty}(u_{\lambda_n},v_{\mu_n}).
\]

The following theorem deals with regularity properties for solutions of \eqref{MS}. The main idea is to apply an appropriate test function in the stability inequality (see Lemma \ref{03} below). This inequality is the main trick to tackle the problem for the case of systems and fourth order equations. This kind of argument involving stability inequality and Moser's iteration method has been used by M.~Crandall and P.~Rabinowitz \cite{CRARAB} and was originated in Harmonic maps and differential geometry.

\begin{theorem}\label{08}
	Assume that $f,g=1$. Then the extremal solution $(u^*,v^*)$ of System $(S_{\lambda^*,\mu^*})$ is smooth when $N\leq 7.$
\end{theorem}

\begin{remark}
	Observe that Theorem \ref{08} determines the critical dimension for this class of Lane-Emden systems, precisely determine the dimension $N^*$ such that the extremal solution is smooth when $N< N^*$ and singular when $N\geq N^*$. Indeed, if we consider $\Omega$ to be the unit ball, $u=v$ and $\lambda=\mu$, then the system turns into a scalar equation and the optimal results are known. For instance, the function $u^*(x)=1-|x|^{2/3}$ is a singular solution for $-\Delta u=\lambda/(1-u)^2$ if $N\geq 8$ (see \cite[Theorem 1.3]{GHOGUO}).
\end{remark}

\subsection{Outline}
\hspace{0.6cm}This paper is organized as follows. In the next section we bring some auxiliary results used in the text. Moreover, we study the existence of a critical curve, extremal parameter and minimal solutions. We also establish upper and lower bounds for the critical curve $\Gamma$ and monotonicity results for the extremal parameter. In Section~\ref{09} we obtain some estimates and properties for the branch of minimal solutions that allow us to prove the regularity result about the extremal solution.

\section{A critical curve: existence of classical solutions}
\hspace{0.6cm}The main goal of this section is to prove Theorems \ref{04}, \ref{06} and \ref{07}. Precisely, by the method of sub-super solutions we prove that there exists a non-increasing continuous function $\Gamma$ of the parameter $\lambda$ such that \eqref{MS} has at least one solution for $0 < \mu < \Gamma(\lambda)$ whereas \eqref{MS} has no
solutions for $\mu>\Gamma(\lambda)$. In what follows unless otherwise stated, by solution we mean a classical solution of class $\mathcal{C}^{2}(\Omega)$. For the sake of completeness, we briefly sketch the proofs of the next lemmas. For more details we refer the reader \cite{COW2011,COWFAZ2014,MM}

\begin{lemma}\label{SS}
	Let $\lambda$ and $\mu$ positive parameters such that there exists a classical super solution $(U,V)$ for \eqref{MS}. Then there exists a classical solution $(u,v)$ of \eqref{MS} such that $u\leq U$ and $v\leq V$.
\end{lemma}

\begin{proof}
	Setting $(u_0,v_0)=(U,V)$
	we can define $(u_n,v_n)$ inductively as follows
	\[
	\left\{
	\begin{alignedat}{3}
	-\Delta u_n = & \, \frac{\lambda f(x)}{(1-v_{n-1})^2} & \quad \text{in} & \quad\Omega,\\
	-\Delta v_n  = & \, \frac{\mu g(x)}{(1-u_{n-1})^2} & \quad \text{in} & \quad\Omega,\\
	0\leq u_n &, v_n < 1 & \quad \text{in} & \quad \Omega,\\
	u_n = v_n & =  \, 0  &  \text{on} & \quad \partial\Omega.\\
	\end{alignedat}
	\right.
	\]
	By the maximum principle, we have $0<u_{n}\leq u_{n-1} \leq \ldots u_1\leq u_0$ and $0<v_{n}\leq v_{n-1} \leq\ldots v_1 \leq u_0$. Thus, there exists $(u,v)$ such that $0\leq  u=\lim_{n\rightarrow\infty} u_n \leq U <1$ and $0\leq v=\lim_{n\rightarrow\infty} v_n \leq V <1$ and by a standard compactness argument we have that the above convergence holds in $C^{2,\alpha}(\overline{\Omega})$ to a  solution $(u,v)$ of \eqref{MS} and in particular different from zero.
\end{proof}

\hspace{0.6cm}We now state and prove a monotonicity result on the coordinates of a solution of \eqref{MS}, precisely,

\begin{lemma}\label{10}
	Suppose that $(u,v)$ is a smooth solution of \eqref{MS} where $0<\mu\leq\lambda.$ Then $\mu u/\lambda\leq v\leq u$ a.e in  $\Omega$.
\end{lemma}

\begin{proof}
	Take the difference of the equations in \eqref{MS}, multiplying this equation by $(u-v)_-$ and integrating by parts we have
	\[
	\int_{\Omega}|\nabla (u-v)_{-}|^2\,d x = \int_{\Omega}\left(\frac{\lambda}{(1-v)^2}-\frac{\mu}{(1-u)^2}\right)(u-v)_{-} \, d x.
	\]
	Since the right hand side is nonpositive and the left hand side is nonnegative, we see that $(u-v)_-=0$ a.e. in $\Omega$ and so $u\geq v$ a.e. in $\Omega$. Now, since $u\geq v$,
	\[
	-\Delta\left( v-\frac{\mu}{\lambda}u \right)=\mu\left( \frac{1}{(1-u)^2}-\frac{1}{(1-v)^2} \right)\geq 0.
	\]
	Thus, $\frac{\mu}{\lambda}u\leq v$ and we finish the proof.
\end{proof}

\hspace{0.6cm}We are going to prove that \eqref{MS} has a classical solution for $\lambda$ and $\mu$ sufficiently small, more precisely, the set $\Lambda:= \{ (\lambda,\mu)\in \mathcal{Q}: \text{ \eqref{MS} has a classical solution} \}$ has nonempty interior.

\begin{lemma}
	\label{01} There exists $\lambda_1>0$ such that $(0,\lambda_1]\times (0,\lambda_1] \subset \Lambda$.
\end{lemma}

\begin{proof}
	Let $B_{R}$ be a ball of radius $R$ such that $\Omega\subset B_{R}$ and let $\mu_{1,R}$ be the first eigenvalue of the Dirichlet
	boundary value problem $(-\Delta,\, H_{0}^{1}(\Omega))$ and denote the corresponding eigenfunction by $\psi_{1,R}$ which we may assume to be positive and also that  $\sup_{B_{R}} \psi_{1,R}=1$.
	Now we show that there exists $\theta>0$ such that $\psi=\theta \psi_{1,R} $ is a super-solution to
	$(S_{\lambda,\lambda})$ provided $\lambda>0$ is sufficiently small. Notice that we can choose $\theta \in (0,1)$ such that $ 0< 1-\theta\psi_{1,R} <1$ in $B$. Thus
	\[
	\left\{
	\begin{alignedat}{6}
	-\Delta \psi = & \, \mu_{1,R}\theta\psi_{1,R}\, & \geq & \, \frac{\lambda f(x)}{(1-\psi)^2} = & \, \frac{\lambda f(x)}{(1-\theta \psi_{1,R})^2} \quad \text{in} & \quad \Omega,\\
	-\Delta \psi = & \, \mu_{1,R}\theta\psi_{1,R}\, & \geq & \, \frac{\lambda g(x)}{(1-\psi)^2} = & \, \frac{\lambda g(x)}{(1-\theta \psi_{1,R})^2} \quad \text{in} & \quad \Omega,
	\end{alignedat}
	\right.
	\]
	provided $\mu_{1,R}\theta\psi_{1,R}(1-\theta \psi_{1,R})^2\geq \lambda\max\{f(x),g(x)\}$.
	Notice that $s_1:=\inf_{x\in \Omega}\theta\psi_{1,R} < s_2:=\sup_{x\in \Omega}\theta\psi_{1,R}<1$, and $s_1,\ s_2$ depend
	of $R$. Setting $Z(s):=s(1-s)^2$, it is easy to see that we can choose $\lambda>0$ sufficiently small such that $\mu_{1,R} \inf_{x \in \Omega}Z(\theta \psi_{1,R}(x)) \geq \lambda\max\{\sup_{\Omega} g(x), \ \sup_{\Omega} f(x)\}$.
	Thus, using Lemma \ref{SS} we conclude that $(\lambda,\mu)\in\Lambda$, for all $\lambda,\mu\in (0,\lambda_1]$.
\end{proof}

\begin{lemma}
	\label{bounded} $\Lambda $ is bounded.
\end{lemma}

\begin{proof}
	Let
	$(\lambda,\mu) \in \Lambda$ and $(u,v)$ the corresponding solution of \eqref{MS}. Multiplying the first equation in \eqref{MS} by $\psi_{1,R}$ and integrating by parts implies that
	\[
	|B_R|\mu_{1,R} \geq  \lambda \int_{B_R} f(x) \psi_{1,R} \,\mathrm{d} x.
	\]
	Analogously, multiplying the second equation in \eqref{MS} by $\psi_{1,R}$ we obtain
	\[
	|B_R|\mu_{1,R} \geq  \mu \int_{B_R} g(x) \psi_{1,R} \,\mathrm{d} x
	\]
	and therefore $\Lambda$ is bounded.
\end{proof}

\hspace{0.6cm}Now we state that $\Lambda$ is a convex set, precisely,
\begin{lemma}\label{connected}If $(\lambda^\prime,\mu^\prime)\in\mathcal{Q}$ and $\lambda' \leq \lambda$ and $\mu'\leq\mu $ for some $(\lambda,\mu)\in \Lambda$ then $(\lambda',\mu')\in \Lambda$.
\end{lemma}
\begin{proof}
	It follows from Lemma \ref{SS}. Indeed, the solution associated to the pair $(\lambda,\mu)\in \Lambda$ turns out to be a super-solution to $(S_{\lambda',\mu'})$.
\end{proof}

\subsection{Critical curve}
\hspace{0.6cm}For each fixed $\theta>0$ consider the line $L_{\theta}=\{\lambda>0: (\lambda,\theta\lambda )\in \Lambda\}$.
Observe that Lemma \ref{01} and Lemma \ref{bounded} implies that for each $\theta$ fixed, the line $L_\theta$ is nonempty and bounded. This allow us to define the curve $\Gamma:(0,+\infty)\rightarrow \mathcal{Q}$ by $\Gamma(\theta):= (\lambda^*(\theta),\mu^*(\theta))$ where $\lambda^*(\theta):=\sup L_{\theta}$ and $\mu^*(\theta)=\theta\lambda^*(\theta)$.

\begin{proof}[Proof of Theorem \ref{04}]
	Define $\mathcal{O}_1=\Lambda\setminus \Gamma$. Given $(\lambda_1,\mu_1),(\lambda_2,\mu_2)\in\mathcal{O}_1$, there exist $\theta_1,\theta_2>0$ such that $\mu_1=\theta_1\lambda_1$ and $\mu_2=\theta_2\lambda_2$. We can define, using the Lemma \ref{connected}, a path linking $(\lambda_1,\mu_1)$ to $(0,0)$ and another path linking $(0,0)$ to $(\lambda_2,\mu_2)$. Follows that $\mathcal{O}_1$ is connected. The Lemma \ref{SS} implies that for each $(\lambda,\mu)\in\mathcal{O}_1$ there exists a positive minimal classical solution $(u_\lambda,v_\mu)$ for problem \eqref{MS}. Now, define $\mathcal{O}_2=\mathcal{Q}\setminus \{ \Lambda \cup \Gamma \}$. Let $(\lambda_1,\mu_1),(\lambda_2,\mu_2)\in\mathcal{O}_2$. Take $(\lambda_{\max} ,\mu_{\max})\in\mathcal{O}_2$, where $\lambda_{\max}=\max{\{\lambda_1,\lambda_2\}}$ and $\mu_{\max}=\max{\{\mu_1,\mu_2\}}$. We can take a path linking $(\lambda_1,\mu_1)$ to $(\lambda_{\max} ,\mu_{\max})$ and another path linking $(\lambda_{\max} ,\mu_{\max})$ to $(\lambda_2,\mu_2)$. Follows that $\mathcal{O}_2$ is connected.
\end{proof}

\subsection{Upper and lower bounds for the critical curve}
\hspace{0.6cm}As noticed by N.~Ghoussoub and Y.~Guo \cite{GHOGUO}, the lower bound for the critical parameter is useful to prove existence of solutions for \eqref{E1}. The following lemma will be the main tool to obtain the estimates contained in Theorem \ref{06} and gives more computationally accessible lower estimates for the critical curve.
\begin{lemma}\label{forma1}
	Assume that $\Omega=B=B_R$ and $f,g$ are radial, that is, $f(x)=f(|x|)$ and  $g(x)=g(|x|)$, for
	all $x \in B$. Then
	\[
	(0,a_{(f,R,N)}]\times(0,a_{(g,R,N)}] \subset \Lambda
	\]
	where
	\[
	a_{(f,R,N)}:=C_{N}\frac{1}{\sup_B f(x)}\frac{1}{R^2},\;\;
	a_{(g,R,N)}:=C_{N}\frac{1}{\sup_B g(x)}\frac{1}{R^2}.
	\]
	and
	\[
	C_{N}=\max\left\{\frac{8N}{27},\frac{6N-8}{9} \right\}.
	\]
\end{lemma}

\begin{proof}
	Notice that the function $w(x):=1/3\left(1-|x|^2/R^2\right)$ satisfies
	\[
	-\Delta w  = \frac{2N}{3R^2} \geq  \frac{8N}{27R^2 \sup_B f}\frac{f(x)}{\left[1-\frac{1}{3}\left(1-\frac{|x|^2}{R^2}\right)\right]^2} = \frac{8N}{27R^2 \sup_B f}\frac{f(x)}{\left(1-w\right)^2}.
	\]
	Similarly,
	\[
	-\Delta w  \geq  \frac{8N}{27R^2 \sup_B
		g}\frac{g(x)}{\left(1-w\right)^2}.
	\]
	Thus, for $\lambda\leq 8N/(27R^2 \sup_B f)$ and $\mu \leq 8N/(27R^2 \sup_B g)$ we have that $(w,w)$ is a super-solution of (\ref{MS}) in $B$.
	Similarly, we can see that, taking $v(x):=1-\left(|x|/R\right)^{2/3}$, the pair $(v,v)$ is a super-solution for (\ref{MS}) in $B$ provided that $\lambda\leq (6N-8)/(9R^2 \sup_B f)$ and $\mu \leq (6N-8)/(9R^2 \sup_B g)$, which completes the proof.
\end{proof}

\begin{proposition}Assume that $\Omega=B=B_R$ and $f(x)=|x|^\alpha,\;\; g(x)=|x|^\beta$ with $\alpha,\ \beta \geq
	0$, then
	\[
	(0,a_{(\alpha,R,N)}]\times(0,b_{(\beta,R,N)}] \subset \Lambda,
	\]
	where
	\[
	a_{(\alpha,R,N)}:=
	\max\left\{\frac{4(2+\alpha)(N+\alpha)}{27},\frac{(2+\alpha)(3N+\alpha-4)}{9}
	\right\}\frac{1}{R^{2+\alpha}}
	\]
	and
	\[
	b_{(\beta,R,N)}:=
	\max\left\{\frac{4(2+\beta)(N+\beta)}{27},\frac{(2+\beta)(3N+\beta-4)}{9}
	\right\}\frac{1}{R^{2+\beta}}.
	\]
\end{proposition}

\begin{proof}
	Consider the function $w_{(\alpha,R)}(x)=1/3\left(1-|x|^{2+\alpha}/R^{2+\alpha}\right)$. Using a similar computation as we have done in the previous lemma we can prove that the pair $(w_{(\alpha,R)},w_{(\beta,R)})$ is a
	super-solution of (\ref{MS}) in $B$ provided that
	\[
	\lambda \leq \frac{4(2+\alpha)(N+\alpha)}{27R^{2+\alpha}} \mbox{ and
	} \mu \leq \frac{4(2+\beta)(N+\beta)}{27R^{2+\beta}}.
	\]
	The same holds for the function $w(x)=1-\left(|x|/R\right)^{(2+\alpha)/3}$ if 
	\[
	\lambda \leq \frac{(2+\alpha)(3N+\alpha-4)}{9R^{2+\alpha}} \mbox{
		and } \mu \leq \frac{(2+\beta)(3N+\beta-4)}{9R^{2+\beta}}.
	\]
\end{proof}

\begin{proof}[Proof of Theorem \ref{07}]
	Consider $(\lambda,\mu) \in \Lambda$ and $(u,v)$ the corresponding solution of (\ref{MS}). Let $\mu_{1}$ and denote the corresponding positive eigenfunction by $\psi_1$. Taking $\psi_{1}$ as a test function in the first equation of
	(\ref{MS}) and using integration by parts we obtain
	\[
	\int_{\Omega} \left(-\mu_{1} u + \frac{\lambda f(x)}{(1-v )^2}
	\right)\psi_{1} \,\,\mathrm{d} x =0
	\]
	which implies that $\lambda > \lambda^* $ when
	\begin{equation}\label{Mall}
	-\mu_{1} u + \frac{\lambda f(x)}{(1-v)^2} >0 \mbox{ in }
	\Omega.
	\end{equation}
	After a simple calculation we find that (\ref{Mall}) holds when
	\[
	\lambda > \frac{4\mu_1}{27}\frac{1}{\inf_\Omega f(x)}.
	\]
	Using the same approach in the second equation we finish the proof.
\end{proof}

\subsection{Monotonicity results for the extremal parameter}
\hspace{0.6cm}Let $G_\Omega(x,\xi)=G(x,\xi)$ be the Green's function of the Laplace operator for the region $\Omega$, with $G(x,\xi)=0$ if $x\in \partial\Omega$. We shall write $(u_{n,\Omega}(x),v_{n,\Omega}(x))=(u_{n}(x),v_{n}(x))$ for the sequence obtained by the interaction process as follows: $(u_0,v_0)=(0,0)$ in $\Omega$ and
\begin{equation}\label{seqe}
\left\{
\begin{alignedat}{4}
u_n(x) & =  \,\int_\Omega\frac{\lambda f(x)G(x,\xi)}{(1-v_{n-1})^2} \ d\xi & \text{ in } & \Omega, \\
v_n(x) & =  \int_\Omega\frac{\mu g(x)G(x,\xi)}{(1-u_{n-1})^2}  \ d \xi & \text{ in } & \Omega. \\
\end{alignedat}
\right.
\end{equation}
It is easy to see that the sequence above converges uniformly for a minimal solution of (\ref{MS}) provided that
$0<\lambda<\lambda^\ast$ and $0<\mu<\Gamma(\lambda)$. This construction will help us to prove the monotonicity result for $\lambda^*$ stated in Theorem \ref{pitomba}.

\begin{proof}[Proof of Theorem \ref{pitomba}]
	Let $(u_{n,\Omega'},v_{n,\Omega'})$ be defined as in (\ref{seqe}) with $\Omega$ replaced by $\Omega'$. Using the corresponding Green's functions for the subdomains $\Omega' \subset\Omega$ satisfy the inequality $G_{\Omega'}(x,\xi)\leq G_\Omega(x,\xi)$ we have
	\[
	\begin{alignedat}{3}
	u_{1,\Omega'}(x) & = & \int_{\Omega'}\lambda f(x)G_{\Omega'}(x,\xi) d \xi
	\leq  \int_\Omega \lambda f(x)G_{\Omega}(x,\xi) \ d\xi  \text{ in }  \Omega', \\
	v_{1,\Omega'}(x) & = & \int_{\Omega'} \mu g(x)G_{\Omega'}(x,\xi) \ d \xi
	\leq  \int_\Omega\mu g(x)G_{\Omega}(x,\xi) \ d\xi  \text{ in } \Omega'. \\
	\end{alignedat}
	\]
	By induction we conclude that $u_{n,\Omega'}(x) \leq u_{n,\Omega}(x)$ and $v_{n,\Omega'}(x)\leq v_{n,\Omega}(x)$ in $\Omega'$. On the other hand, since $u_{n,\Omega}(x)\leq u_{n+1,\Omega}(x)$ and $v_{n,\Omega}(x)\leq  v_{n+1,\Omega}(x)$ in $\Omega$, for each $n$, we get that $u_{n,\Omega'}(x) \leq u_{\Omega}(x)$ and $v_{n,\Omega'}(x) \leq v_{\Omega}(x)$ in $\Omega'$ and we are done.
\end{proof}

\begin{corollary}
	Suppose $f_1,f_2,g_1,g_2: \overline{\Omega}\rightarrow \mathbb{R}$ satisfy condition \eqref{05} and $f_1(x)\leq f_2(x)$ and $g_1(x) \leq g_2(x)$ for all $x\in\Omega$, then $\lambda^\ast(f_1,g_1) \geq \lambda^\ast(f_2,g_2)$ and for each $\lambda \in (0,\lambda^\ast(f_2,g_2))$. Furthermore $u_{1}(x) \leq u_{2}(x)$ and $v_{1}(x)\leq v_{2}(x)$ for all $x\in\Omega$ for the corresponding minimal solutions. If $f_1(x) < f_2(x)\mbox{ or } g_1(x) < g_2(x)$ on a subset of positive measure, then
	$ u_{1}(x) < (u_{2}(x) $ and $v_{1}(x) < v_{2}(x) $  for all $ x \in\Omega$.
\end{corollary}

\hspace{0.6cm}We shall use Schwarz symmetrization method. Let $B_R=B_R(0)$ the Euclidean ball in $\mathbb{R}^N$ with radius $R>0$ centred at origin such that $|B_R|$=$|\Omega|$, and let $u^\sharp$ be the symmetrization of $u$, then it is well-known that $u^\sharp$ depends only on $|x|$ and $u^\sharp$ is a decreasing function of $|x|$.

\begin{proof}[Proof of Theorem \ref{monoto1}]
	For each $\lambda \in (0,\lambda^\ast(B_R,f,g))$ and $\mu\in (0,\Gamma_{(B_R,f^\sharp,g^\sharp)}(\lambda))$ we consider the minimal sequence $(u_n,v_n)$ for (\ref{MS}) as defined in (\ref{snow1}), and let $(\widehat{u}_n,\widehat{v}_n)$ be the minimal sequence for the corresponding Schwarz symmetrized problem:
	\begin{equation}\label{snow2}
	\left\{
	\begin{alignedat}{4}
	-\Delta u & = \frac{\lambda f^\sharp(x)}{(1-v)^2}  & \text{ in } & B_R, \\
	-\Delta v & = \frac{\mu g^\sharp(x)}{(1-u)^2}  & \text{ in } & B_R, \\
	0< & u,v < 1 & \text{ in } & B_R , \\
	u &= v  =  0 & \text{ on } & \partial B_R.
	\end{alignedat}
	\right.
	\end{equation}
	Since $\lambda \in (0,\lambda^\ast(B_R,f,g))$ and $\mu\in (0,\Gamma_{(B_R,f^\sharp,g^\sharp)}(\lambda))$ we can consider the corresponding minimal solution $(\widehat{u},\widehat{v})$ of (\ref{snow2}). As in the proof of Lemma \ref{Peroba} we have $ 0< \widehat{u}_n \leq\widehat{u} <1$ and $0<\widehat{v}_n\leq \widehat{v}<1$ on $B_R$ for all $n$. We shall prove for the sequence $(u_n,v_n)$ we also have $0< u^\sharp_n \leq \widehat{u} <1$ and $0<v^\sharp_n\leq \widehat{v}<1$ on $B_R$ for all $n$. Therefore, the minimal sequence $(u_n,v_n)$ for (\ref{MS}) satisfies $u_n(x)\leq \max_{x\in B_R} \widehat{u}$ and $v_n(x)\leq \max_{x\in B_R} \widehat{v}$ and again as in the proof of Lemma \ref{SS}, there exists a minimal solution $(u,v)$ for (\ref{MS}).
\end{proof}

\begin{proof}[Proof of Theorem \ref{06}]
	Since $\sup_{B_R}f^\sharp = \sup_\Omega f$ and $\sup_{B_R}g^\sharp = \sup_\Omega g$, setting $R=\left(|\Omega|/\omega_N\right)^{1/N}$ the proof follows as an applications of Theorem \ref{monoto1}
	and Lemma \ref{forma1}.
\end{proof}

\section{The branch of minimal solutions}\label{09}
\hspace{0.6cm}Next, assuming the existence of solutions for System \eqref{MS}, we obtain also existence and uniqueness of minimal solution.

\begin{lemma}\label{Peroba}
	For any $0<\lambda<\lambda^\ast$ and $0<\mu<\Gamma(\lambda)$, there exists a unique minimal solution
	$(u,v)$ of (\ref{MS}).
\end{lemma}
\begin{proof}
	This minimal solution is obtained as the limit of the sequence of pair of functions $(u_n,v_n)$ constructed recursively as
	follows: $(u_0,v_0)=(0,0)$ in $\Omega$ and for each $n=1,2,\ldots$, $(u_n,v_n)$ is the unique solution of the boundary value problem:
	\begin{equation}\label{snow1}
	\left\{
	\begin{alignedat}{4}
	-\Delta u_n & =  \frac{\lambda f(x)}{(1-v_{n-1})^2}  & \text{ in } & \Omega, \\
	-\Delta v_n & =  \frac{\mu g(x)}{(1-u_{n-1})^2}  & \text{ in } & \Omega, \\
	0 < &u_n, v_n <1 & \text{ in } & \Omega , \\
	u_n &= v_n  =  0  & \text{ on } &
	\partial \Omega.
	\end{alignedat}
	\right.
	\end{equation}
	Let $(U,V)$ be any solution for problem (\ref{MS}). First, it is clear that  $1 \geq U > u_0\equiv 0$ and   $1 \geq V > v_0\equiv 0$ in $\Omega$. Now, assume that $U \geq u_{n-1}$ and  $V \geq v_{n-1}$ in $\Omega$. Thus,
	\[
	\left\{
	\begin{alignedat}{4}
	-\Delta (U-u_n) & =  \lambda f(x)\left[\frac{1}{(1-V)^2}-\frac{1}{(1-u_{n-1})^2}\right] \geq 0& \text{ in } & \Omega ,\\
	-\Delta (V-v_n) & =  \mu g(x)\left[\frac{1}{(1-U)^2}-\frac{1}{(1-v_{n-1})^2}\right] \geq 0& \text{ in } & \Omega ,\\
	U-u_{n} & =  V-v_n = 0&   \text{ on } & \partial \Omega.
	\end{alignedat}
	\right.
	\]
	By the maximum principle we conclude that $1>U\geq u_n>0 \mbox{ and } 1>V\geq v_n>0 \mbox{ in } \Omega$. It is clear that this kind of argument implies that $(u_n,v_n)$ is a monotone increasing sequence. Therefore, $(u_n,v_n)$ converges uniformly to a solution $(u,v)$ of (\ref{MS}), which by construction is unique in this class of minimal solutions.
\end{proof}

\hspace{0.6cm}We can introduce for any solution $u$ of \eqref{E1}, the linearized operator at $u$ defined by $L_{u, \lambda} =-\Delta-\frac{2\lambda f(x)}{(1-u)^3}$ and its eigenvalues $\{\mu_{k,\lambda}(u); k=1,2,...\}$. The first eigenvalue is then
simple and  can be characterized variationally by
\[
\mu_{1,\lambda}(u)= \inf \left\{ \left\langle L_{u,\lambda}
\phi,\phi \right\rangle_{H_0^1(\Omega)}; \, \phi \in
C_0^\infty(\Omega),  \int_\Omega  |\phi (x)|^2dx =1\right\}.
\]
{\it Stable solutions} (resp.,  {\it semi-stable solutions}) of $(S)_\lambda$ are those solutions $u$ such that
$\mu_{1,\lambda}(u)>0$ (resp., $\mu_{1,\lambda}(u)\geq 0$). Following the ideas of M.~Crandall and P.~Rabinowitz \cite{CRARAB}, it was shown in \cite{GHOGUO} that for $1\leq N\leq 7$ and for $\lambda$ close enough to $\lambda^*$ that there exists a unique second branch of solutions for \eqref{E1} bifurcating  from $u^*$.

In the case that $(u,v)$ is a solution of (\ref{MS}) we consider the first eigenvalue $\nu_1=\nu_1((\lambda,\mu),(u,v))$ of the
linearization $\mathfrak{L}:=-\overrightarrow{\Delta}-A(x)$ around $(u,v)$ under Dirichlet boundary conditions,  where
\[
\overrightarrow{\Delta}\Phi=\left(
\begin{array}{c}
\Delta \phi_1 \\
\Delta \phi_2 \\
\end{array}
\right)
\]
and
\[ A(x):=\left(
\begin{array}{cc}
0 & a_{1 2}(x) \\
a_{2 1}(x) & 0 \\
\end{array}
\right)
=
\left(
\begin{array}{cc}
0 & \frac{2\lambda f(x)}{(1-v(x))^3} \\
\frac{2 \mu g(x)}{(1-u(x))^3} & 0 \\
\end{array}
\right)
\]
that is, the eigenvalue problem
\[
\mathfrak{L}\Phi = \nu \Phi, \;\; \Phi \in W^{1,2}_0(\Omega)\times
W^{1,2}_0(\Omega),
\]
namely, $\nu_1$ is the first eigenvalue of the problem
\begin{equation}
\left\{
\begin{alignedat}{4}
-\Delta \phi_1 - \frac{2\lambda f(x)}{(1-v)^3}\phi_2 &= \nu \phi_1  & \text{ in } & \Omega, \\
-\Delta \phi_2 - \frac{2\mu g(x)}{(1-u)^3}\phi_1 &=\nu\phi_2 & \text{ in } & \Omega, \\
\phi_1 = \phi_2  & = 0 & \text{ on } & \partial \Omega.
\end{alignedat}
\right. \tag{$E_{(\lambda, \mu)} $} \label{Eig1}
\end{equation}

We recall that in \cite[Proposition~3.1]{Sweers} was proved that there exists a unique eigenvalue $\nu_1$ with strictly positive
eigenfunction $\phi=(\phi_1,\phi_2)$ of (\ref{Eig1}), that is, $\phi_i>0$ in $\Omega$ for $i=1,2$.

\begin{remark}The first eigenvalue of the linearized single equation has a variational characterization; no such analogous formulation is available for our system.
\end{remark}

\begin{definition}[Stable and Semi-stable Solution] A solution of (\ref{MS}) is said to be stable
	(resp. semi-stable) if $\nu_1 >0 $ (resp., $\nu_1\geq 0$).
\end{definition}

\begin{proposition}
	Suppose that $(\lambda,\mu)\in\Lambda$ with $0<\mu\leq\lambda$ and we let $(u,v)$ denote the minimal solution of \eqref{MS}. Let $\phi_1,\phi_2$ as in \eqref{Eig1}. Then
	\[
	\frac{\phi_2}{\phi_1}\geq\frac{\mu}{\lambda}\quad \text{ in }\Omega.
	\]
\end{proposition}

\begin{proof}
	Take the difference equation in \eqref{Eig1} and use Lemma \ref{10} to obtain
	\[
	-\Delta(\phi_2-\phi_1)-\nu(\phi_2-\phi_1)+\frac{\mu(\phi_2-\phi_1)}{(1-v)^3}\geq \frac{(\mu-\lambda)\phi_2}{(1-v)^3} \quad \text{ in }\Omega.
	\]
	Now, define a elliptic operator $L:=-\Delta -\nu$. We have that
	\[
	\begin{alignedat}{2}
	L\left( \psi_2-\psi_1 +\frac{\lambda-\mu}{\lambda}\psi_1 \right)&+\frac{\mu}{(1-v)^3}\left( \psi_2-\psi_1+\frac{\lambda-\mu}{\lambda} \right)\\
	& \geq L\left( \psi_2-\psi_1 +\frac{\lambda-\mu}{\lambda}\psi_1 \right)+\frac{\mu}{(1-v)^3}\left( \psi_2-\psi_1 \right)\\
	& \geq \frac{(\mu-\lambda)\phi_2}{(1-v)^3}+\frac{\lambda-\mu}{\lambda}L(\phi_1)=0
	\end{alignedat} 
	\]
	Using the maximum principle, we have $\phi_2-\phi_1+(\lambda-\mu)\phi_1/\lambda\geq 0\text{ in }\Omega$. Re-arranging the above equation follows $\phi_2/\phi_1\geq\mu/\lambda$ and this finish the proof.
\end{proof}

\subsection{Estimates for minimal solutions}
\hspace{0.6cm}The next result is crucial in our argument to obtain the regularity of semistable solutions of \eqref{MS}. For the proof we refer the reader to \cite[Lemma~3]{DUPFARSIR2013}.

\begin{lemma}\label{03}
	Let $N\geq 1$ and let $(u,v)\in C^2\left(\overline{\Omega}\right)\times C^2\left(\overline{\Omega}\right)$ denote a stable solution of
	\[
	\left\{
	\begin{alignedat}{4}
	-\Delta u  &= g(v)  & \text{ in } & \Omega, \\
	-\Delta v &= f(u) & \text{ in } & \Omega, \\
	u = v  & = 0 & \text{ on } & \partial \Omega,
	\end{alignedat}
	\right.
	\]
	where $f$ and $g$ denote two nondecreasing $C^1$ functions. Then for all $\varphi\in C_{c}^{1}(\Omega)$ there holds
	\[
	\int_{\Omega}\sqrt{f^\prime(u)g^\prime(v)}\varphi^2\,\mathrm{d}x\leq \int_{\Omega}|\nabla \varphi|^2\,\mathrm{d}x.
	\]
\end{lemma}

\hspace{0.6cm}Now we follow the approach due to L.~Dupaigne, A.~Farina and B.~Sirakov \cite{DUPFARSIR2013} adapted to MEMS case. The main idea is apply H\"older's inequality and iterate both equations in System \eqref{MS}. This method is the key to obtain the optimal dimension for the regularity of extremal solutions.

\begin{proof}[Proof of Theorem \ref{08}]
	Let $\alpha>1$, multiply the equation $-\Delta u=\lambda/(1-v)^2$ by $(1-u)^{-\alpha}-1$ and integrating by parts we have
	\[
	\begin{alignedat}{2}
	\lambda\int_{\Omega}(1-v)^{-2}[(1-u)^{-\alpha}-1]\,\mathrm{d}x& = \alpha\int_{\Omega}(1-u)^{-\alpha-1}|\nabla u|^2\,\mathrm{d}x\\
	& =\frac{4\alpha}{(\alpha-1)^2}\int_{\Omega}\left\vert \nabla \left( (1-u)^{-\frac{\alpha}{2}+\frac{1}{2}}  \right) \right\vert^2\,\mathrm{d}x.
	\end{alignedat}
	\]
	We can test $(1-u)^{-\alpha/2+1/2}-1$ in Lemma \ref{03} to obtain
	\[
	2\sqrt{\lambda\mu}\int_{\Omega}(1-u)^{-\frac{3}{2}}(1-v)^{-\frac{3}{2}}[(1-u)^{-\frac{\alpha}{2}+\frac{1}{2}}-1]^2\,\mathrm{d}x\leq\int_{\Omega}\left\vert \nabla\left( (1-u)^{-\frac{\alpha}{2}+\frac{1}{2}} \right) \right\vert^2\,\mathrm{d}x.
	\]
	Combining these two previous inequalities and developing the square follows
	\begin{equation}\label{26}
	\begin{alignedat}{2}
	\sqrt{\lambda\mu}\int_{\Omega}(1-u)^{\frac{-2\alpha-1}{2}}(1-v)^{-\frac{3}{2}}& \leq \frac{\lambda\left(\alpha-1\right)^2}{8\alpha}\int_{\Omega}(1-u)^{-\alpha}(1-v)^{-2}\\
	&+2\sqrt{\lambda\mu}\int_{\Omega}(1-u)^{\frac{-\alpha-1}{2}}(1-v)^{-\frac{3}{2}}
	\end{alignedat}
	\end{equation}
	Denote
	\[
	X=\int_{\Omega}(1-u)^{\frac{-2\alpha-1}{2}}(1-v)^{-\frac{3}{2}}
	\quad \mbox{and} \quad 
	Y=\int_{\Omega}(1-u)^{-\alpha-1}(1-v)^\frac{-\alpha-3}{2}.
	\]
	Now we need estimate the terms on the right-hand side. Take $p=\alpha/(\alpha-1)$ and $q=\alpha$ and using H\"older inequality with this exponents we obtain
	\begin{equation}\label{25}
	\int_{\Omega}(1-u)^{-\alpha}(1-v)^{-2}\leq X^\frac{\alpha-1}{\alpha}Y^\frac{1}{\alpha}.
	\end{equation}
	Given $\epsilon>0$, we can use Young's inequality and Lemma \ref{10} to obtain
	\begin{equation}\label{24}
	\int_{\Omega}(1-u)^\frac{-\alpha-1}{2}(1-v)^{-\frac{3}{2}}\leq \frac{\epsilon}{2}\frac{\sqrt{\lambda}}{\sqrt{\mu}} \int_{\Omega}(1-u)^{-\alpha}(1-v)^{-2} + \frac{\sqrt{\mu}}{\sqrt{\lambda}}\frac{|\Omega|}{2\epsilon}.
	\end{equation}
	Thus, by \eqref{26},\eqref{25} and \eqref{24} we obtain
	\[
	\sqrt{\lambda\mu}X\leq \lambda\left[\frac{\left(\alpha-1\right)^2}{8\alpha}+\epsilon\right]X^\frac{\alpha-1}{\alpha}Y^\frac{1}{\alpha} + \frac{|\Omega|}{2\epsilon}.
	\]
	By symmetry, we also have
	\[
	\sqrt{\lambda\mu}Y\leq \mu\left[\frac{\left(\alpha-1\right)^2}{8\alpha}+\epsilon\right]Y^\frac{\alpha-1}{\alpha}X^\frac{1}{\alpha} + \frac{|\Omega|}{2\epsilon}.
	\]
	Multiplying this equations we have
	\[
	\begin{alignedat}{2}
	\left[ 1-\left(\frac{\left(\alpha-1\right)^2}{8\alpha}+\epsilon\right)^2 \right]XY & \leq \left[\frac{\left(\alpha-1\right)^2}{8\alpha}+\epsilon\right]\frac{|\Omega|}{2\epsilon}\left[X^{\frac{\alpha-1}{\alpha}}Y^\frac{1}{\alpha} + X^\frac{1}{\alpha}Y^\frac{\alpha-1}{\alpha}\right]\\
	& + \frac{|\Omega|^2}{4\epsilon^2}.
	\end{alignedat}
	\]
	Choose $\epsilon=1/16$ and thus we can verify that for every $1<\alpha<9,62$, either $X$ or $Y$ must be bounded. We can suppose $\lambda\leq\mu$ and by Lemma \ref{10} we have $u\leq v$. Thus follows that $(1-u)^{-3}$ must be bounded, either in $L^p$ for $p<(\alpha+2)/3$ or in $L^q$ for $q<\alpha +5/3$. We note that the second case does not occur, because otherwise the semistable solutions should be regular for dimension $N\leq 22$, but we already known that, in the scalar case, $u^*(x)=1-|x|^{2/3}$ is a singular solution when $\Omega$ is the unit ball and $N\geq 8$. Therefore, the first case must occur and consequently $u^*$ is smooth for $N\leq 7$.
\end{proof}

\begin{remark} Using a result due to W.~Troy \cite[Theorem~1]{troy}, we can see that any smooth solution of \eqref{MS} is radially symmetric and decreasing when $\Omega$ is a ball of $\mathbb{R}^N$.
\end{remark}

\end{document}